\documentclass[11pt]{amsart}
\usepackage[usenames,dvipsnames,svgnames,table]{xcolor}
\usepackage{amssymb,amsfonts,amsmath,amscd}
\usepackage{verbatim}
\usepackage[all]{xy}
\usepackage{hyperref}

\newtheorem{theorem}{Theorem}[section]
\newtheorem{lemma}[theorem]{Lemma}     
\newtheorem{corollary}[theorem]{Corollary}
\newtheorem{proposition}[theorem]{Proposition}

\theoremstyle{definition}

\newtheorem{definition}[theorem]{Definition}

\newtheorem{reminder}[theorem]{Reminder}
\newtheorem{remark}[theorem]{Remark}

\newtheorem{example}[theorem]{Example}
\newtheorem{notation}[theorem]{Notation} 

\newtheorem{finalcomments}[theorem]{Final comments}
 
\newtheorem*{notationsandconventions}{Notations and conventions}

\newcommand{\espai}{\phantom{++}}
\newcommand{\mbn}{\mbox{$\mathbb{N}$}}
\newcommand{\mbz}{\mbox{$\mathbb{Z}$}}
\newcommand{\mbq}{\mbox{$\mathbb{Q}$}}

\newcommand{\mbf}{\mbox{$\mathbb{F}$}}

\newcommand{\mcc}{\mbox{$\mathcal{C}$}}

\newcommand{\mro}{\mbox{$\mathrm{o}$}}

\newcommand{\mfm}{\mbox{$\mathfrak{m}$}}
\newcommand{\mfn}{\mbox{$\mathfrak{n}$}}
\newcommand{\mfp}{\mbox{$\mathfrak{p}$}}
\newcommand{\mfq}{\mbox{$\mathfrak{q}$}}

\newcommand{\id}{\mbox{${\rm id}$}}
\newcommand{\ud}{\mbox{${\rm d}$}}

\newcommand{\mdp}{\mbox{${\rm m}$}}

\newcommand{\charac}{\mbox{${\rm char}$}}

\newcommand{\height}{\mbox{${\rm ht}$}}

\newcommand{\rank}{\mbox{${\rm rank}$}}

\newcommand{\Spec}{\mbox{${\rm Spec}$}}
\newcommand{\Max}{\mbox{${\rm Max}$}}

\title[Integral degree of integral ring extensions]
{On the integral degree of integral ring extensions}

\author{Jos\'e M. Giral, Liam O'Carroll, Francesc Planas-Vilanova,
  and Bernat Plans}

\date{\today}

\subjclass[2010]{13B21,13B22,13G05,12F05}

\begin{document}

\begin{abstract}
Let $A\subset B$ be an integral ring extension of integral domains
with fields of fractions $K$ and $L$, respectively.  The integral
degree of $A\subset B$, denoted by $\ud_A(B)$, is defined as the
supremum of the degrees of minimal integral equations of elements of
$B$ over $A$. It is an invariant that lies in between $\ud_K(L)$ and
$\mu_A(B)$, the minimal number of generators of the $A$-module
$B$. Our purpose is to study this invariant. We prove that it is
sub-multiplicative and upper-semicontinuous in the following three
cases: if $A\subset B$ is simple; if $A\subset B$ is projective and
finite and $K\subset L$ is a simple algebraic field extension; or if
$A$ is integrally closed. Furthermore, $\ud$ is upper-semicontinuous
if $A$ is noetherian of dimension $1$ and with finite integral
closure. In general, however, $\ud$ is neither sub-multiplicative nor
upper-semicontinuous.
\end{abstract}

\maketitle 

\section{Introduction}

Let $A\subset B$ be an integral ring extension, where $A$ and $B$ are
two commutative integral domains with fields of fractions $K=Q(A)$ and
$L=Q(B)$, respectively. Then, for any element $b\in B$, there exist
$n\geq 1$ and $a_{i}\in A$, such that
\begin{eqnarray*}
b^{n}+a_{1}b^{n-1}+a_{2}b^{n-2}+\ldots +a_{n-1}b+a_{n}=0.
\end{eqnarray*}
The minimum integer $n\geq 1$ satisfying such an equation is called
the {\em integral degree of $b$ over $A$} and is denoted by
$\id_A(b)$. The supremum, possibly infinite, of all the integral
degrees of elements of $B$ over $A$, $\sup\{\id_A(b)\mid b\in B\}$, is
called the {\em integral degree of $B$ over $A$} and is denoted by
$\ud_A(B)$.

These notions are indeed very natural. They were explicitely
considered in \cite{gp} and, previously, in a different framework, by
Kurosch \cite{kurosch}, Jacobson \cite{jacobson-annals}, Kaplansky
\cite{kaplansky} and Levitzki \cite{levitzki}, and more recently by
Voight \cite{voight}.

The goal in \cite{gp} was to study the uniform Artin-Rees property
with respect to the set of regular ideals having a principal
reduction. It was proved that the integral degree, in fact, provides a
uniform Artin-Rees number for such a set of ideals.

The purpose of the present paper is to investigate more deeply the
invariant $\ud_A(B)$. We first note that $\ud_A(B)$ is between
$\ud_K(L)$, the integral degree of the corresponding algebraic field
extension $K\subset L$, and $\mu_A(B)$, the minimal number of
generators of the $A$-module $B$. That is,
\begin{eqnarray*}
\ud_K(L)\leq \ud_A(B)\leq\mu_A(B).
\end{eqnarray*} 
In a sense, $\ud_A(B)$ can play the role of, or just substitute for,
one of them. For instance, it is a central question in commutative
ring theory whether the integral closure $\overline{A}$ of a domain
$A$ is a finitely generated $A$-module. It is well-known that, even
for one-dimensional noetherian local domains, $\mu_A(\overline{A})$
might be infinite (see, e.g., \cite[Section~4.9]{hs},
\cite[\S~33]{matsumura-crt}). However, for one-dimensional noetherian
local domains $\ud_A(\overline{A})$ is finite
(\cite[Proposition~6.5]{gp}). Hence, in this situation, $\ud_A(B)$
would be an appropriate substitute for $\mu_A(B)$. Another positive
aspect of $\ud_A(B)$, compared with $\mu_A(B)$, is good behaviour with
respect to inclusion, i.e., if $B_1\subset B_2$, then $\ud_A(B_1)\leq
\ud_A(B_2)$, while in general we cannot deduce that $\mu_A(B_1)$ is
smaller than or equal to $\mu_A(B_2)$.

Similarly, $\ud_K(L)$ is a simplification of $\ud_A(B)$. Note that
$\ud_K(L)\leq [L:K]$, the degree of the algebraic field extension
$K\subset L$. We will see that $\ud_K(L)=[L:K]$ if and only if
$K\subset L$ is a simple algebraic field extension.

Of special interest would be to completely characterise when
$\ud_A(B)$ reaches its maximal or its minimal value. We will say that
$A\subset B$ has {\em maximal integral degree} when
$\ud_A(B)=\mu_A(B)$. Similarly, we will say that $A\subset B$ has {\em
  minimal integral degree} when $\ud_K(L)=\ud_A(B)$. Examples of
maximal integral degree are simple integral extensions $A\subset
B=A[b]$, $b\in B$ (Proposition~\ref{first-prop},~$(b)$). Examples of
minimal integral degree occur when $A\subset B$ is a projective finite
integral ring extension with corresponding simple algebraic field
extension $K\subset L$ or when $A$ is integrally closed
(cf. Theorem~\ref{proj-klsimple} and Proposition~\ref{kronecker}). By
a projective, respectively free, finite ring extension $A\subset B$ we
mean that $B$ is a finitely generated projective, respectively free,
$A$-module. Moreover, integral ring extensions $A\subset B$ of both at
the same time minimal and maximal integral degree are precisely free
finite integral ring extensions $A\subset B$ with corresponding simple
algebraic field extension $K\subset L$ (see
Corollary~\ref{free-klsimple}).

Considering the multiplicativity property of the degree of algebraic
field extensions $K\subset L\subset M$, that is, $[M:K]=[L:K][M:L]$,
and the sub-multiplicativity property of the minimal number of
generators of integral ring extensions $A\subset B\subset C$, namely,
$\mu_A(C)\leq \mu_A(B)\mu_B(C)$, it is natural to ask for the same
property of $\ud_A(B)$.  We will say that the integral degree $\ud$ is
{\em sub-multiplicative with respect to} $A\subset B$ if $\ud_A(C)\leq
\ud_A(B)\ud_B(C)$, for every integral ring extension $B\subset C$. We
prove that $\ud$ is sub-multiplicative with respect to $A\subset B$ in
the following three situations: if $A\subset B$ has maximal integral
degree (e.g., if $A\subset B$ is simple); if $A\subset B$ is
projective and finite and $K\subset L$ is simple; or if $A$ is
integrally closed (see Corollaries~\ref{sm-mid}, ~\ref{sm-projective}
and~\ref{sm-icd}). Note that in the three cases above, $A\subset B$
has either maximal integral degree, or else minimal integral
degree. We do not know an instance in which $\ud$ is
sub-multiplicative with respect to $A\subset B$ and
$\ud_K(L)<\ud_A(B)<\mu_A(B)$. We will prove that $\ud$ is not
sub-multiplicative in general. Taking advantage of an example of
Dede\-kind, we find a non-integrally closed noetherian domain $A$ of
dimension $1$, with finite integral closure $B$, where $B$ is the ring
of integers of a number field, and a degree-two integral extension $C$
of $B$, such that $\ud_A(C)=6$, whereas $\ud_A(B)=2$ and $\ud_B(C)=2$.
In this particular example, $\ud_K(L)=1$, $\ud_A(B)=2$ and
$\mu_A(B)=3$, so $A\subset B$ is neither of maximal nor of minimal
integral degree (see Example~\ref{no-sm}).

Another aspect well worth considering is semicontinuity, taking into
account that the minimal number of generators is an
upper-semicontinuous function (see, e.g., \cite[Chapter~IV, \S~2,
  Corollary~2.6]{kunz}). Note that if $\mfp$ is a prime ideal of $A$,
clearly $A_{\mathfrak{p}}\subset B_{\mathfrak{p}}$ is integral. Thus
one can regard the integral degree as a function
$\ud:\Spec(A)\to\mbn$, defined by
$\ud(\mfp)=\ud_{A_{\mathfrak{p}}}(B_{\mathfrak{p}})$. We will say that
the integral degree $\ud$ is {\em upper-semicontinuous with respect
  to} $A\subset B$ if $\ud:\Spec(A)\to\mbn$ is an upper-semicontinuous
function, that is, if $\{\mfp\in\Spec(A)\mid \ud(\mfp)<n\}$ is open,
for all $n\geq 1$. We prove (in Proposition~\ref{semi-two}) that $\ud$
is upper-semicontinuous with respect to $A\subset B$ in the following
two situations: if $A\subset B$ is simple; or if $A\subset B$ has
minimal integral degree (e.g., if $A\subset B$ is projective and
finite and $K\subset L$ is simple; or if $A$ is integrally closed).
Note that in the two cases above, $A\subset B$ has either maximal
integral degree, or else minimal integral degree. There is a setting
in which we can prove that $\ud$ is upper-semicontinuous with respect
to $A\subset B$, yet $\ud_A(B)$ might be different from $\ud_K(L)$ and
$\mu_A(B)$. This happens when $A$ is a non-integrally closed
noetherian domain of dimension $1$ with finite integral closure (see
Theorem~\ref{semi-nagata}). However, $\ud$ is not upper-semicontinuous
in general, even if $A$ is a noetherian domain of dimension~$1$ (see
Example~\ref{no-semi}).

The paper is organized as follows. In Section~\ref{prelim} we
recall some definitions and known results given in \cite{gp}. We also
prove that $\ud_A(B)$ is a local invariant in the following sense:
\begin{eqnarray*}
\ud_A(B)=\sup\{\ud_{A_{\mathfrak{m}}}(B_{\mathfrak{m}})\mid
\mfm\in\Max(A)\}=\sup\{\ud_{A_{\mathfrak{p}}}(B_{\mathfrak{p}})\mid
\mfp\in\Spec(A)\}. 
\end{eqnarray*}
Observe that the analogue for $\mu_A(B)$ is not true in
general. Section~\ref{sm} is mainly devoted to the
sub-multiplicativity of the integral degree. Sections~\ref{field},
~\ref{proj} and ~\ref{icd} are devoted to the integral degree of,
respectively, algebraic field extensions, projective finite ring
extensions and integral ring extensions with base ring integrally
closed. Finally, Section~\ref{semi} is devoted to the
upper-semicontinuity of the integral degree.

\begin{notationsandconventions}
All rings are assumed to be commutative and with unity.  Throughout,
$A\subset B$ and $B\subset C$ are integral ring extensions. Moreover,
we always assume that $A$, $B$ and $C$ are integral domains, though
many definitions and results can easily be extended to the
non-integral domain case. The fields of fractions of $A$, $B$ and $C$
are denoted by $K=Q(A)$, $L=Q(B)$ and $M=Q(C)$, respectively. The
integral closure of $A$ in $K=Q(A)$ is denoted by $\overline{A}$ and
is simply called the ``integral closure of $A$''. In the particular
case in which $A$, $B$ and $C$ are fields, we write $A=K$, $B=L$ and
$C=M$. Whenever $\{x_1,\ldots ,x_r\}\subset N$ is a generating set of
an $A$-module $N$, we will write $N=\langle x_1,\ldots
,x_r\rangle_A$. The minimal number of generators of $N$ as an
$A$-module, understood as the minimum of the cardinalities of
generating sets of $N$, is denoted by $\mu_A(N)$.
\end{notationsandconventions}

\vspace*{0,7cm}

\noindent {\em Liam O'Carroll, our friend and coauthor, died aged 72
  years, of illness, on October 25, 2017. We greatly miss him.}

\vspace*{0,7cm}

\section{Preliminaries and first properties}\label{prelim}

We start by recalling and extending some definitions and results from
\cite[Section~6]{gp}. Recall that $A\subset B$ is an integral ring
extension of integral domains, and $K=Q(A)$ and $L=Q(B)$ are their fields
of fractions.

\begin{definition}\label{def-int-degree}
Let $b\in B$. A {\em minimal degree polynomial} of $b$ over $A$ (which
is not necessarily unique) is a monic polynomial
$\mdp(T)=T^n+a_1T^{n-1}+\ldots +a_{n-1}T+a_n\in A[T]$, $n\geq 1$, with
$\mdp(b)=0$, and such that there is no other monic polynomial of lower
degree in $A[T]$ and vanishing at $b$.  The {\em integral degree of
  $b$ over $A$}, denoted by $\id_A(b)$, is the degree of a minimal
degree polynomial $\mdp(T)$ of $b$ over $A$. In other words,
\begin{eqnarray*}
\id_{A}(b)=\deg \mdp(T)= \min\{ n\geq 1\mid b\mbox{ satisfies an
  integral equation over $A$ of degree }n\}.
\end{eqnarray*}
The {\em integral degree of $B$ over $A$} is defined as the value
(possibly infinite)
\begin{eqnarray*}
\ud_{A}(B)={\rm sup}\, \{ \id_A(b)\mid b\in B\}.
\end{eqnarray*}
Note that $\ud_{A}(B)=1$ if and only if $A=B$. 
\end{definition}

We give a first example, which will be completed subsequently (see
Corollary~\ref{galois}).

\begin{example}\label{ring-of-invariants}
Let $B$ be an integral domain and let $G$ be a finite group acting as
automorphisms on $B$. Let $A=B^G=\{b\in B\mid \sigma(b)=b,\mbox{ for
  all }\sigma\in G\}$. Then $A\subset B$ is an integral ring extension
and $\ud_A(B)\leq \mro(G)$, the order of $G$.
\end{example}
\begin{proof}
Let $G=\{\sigma_1,\ldots ,\sigma_n\}$. For every $b\in B$, take
$p(T)=(T-\sigma_1(b))\cdots (T-\sigma_n(b))$. Clearly $p(T)\in A[T]$
and $p(b)=0$. Thus $b$ is integral over $A$ and $\id_A(b)\leq
n=\mro(G)$.
\end{proof}

The following is a first list of properties of the integral degree
mainly proved in \cite{gp}.

\begin{proposition}\label{first-prop}
Let $A\subset B$ be an integral ring extension. The following
properties hold.
\begin{itemize}
\item[$(a)$] $\ud_A(B)\leq \mu_A(B)$.
\item[$(b)$] If $A\subset B=A[b]$ is simple, then
  $\id_A(b)=\ud_A(B)=\mu_A(B)$.
\item[$(c)$] If $S$ is a multiplicatively closed subset of $A$, then
  $S^{-1}A\subset S^{-1}B$ is an integral ring extension and
  $\ud_{S^{-1}A}(S^{-1}B)\leq \ud_{A}(B)$.
\item[$(d)$] If $S=A\setminus \{0\}$, then $S^{-1}B=L$.
\item[$(e)$] $\ud_K(L)\leq\ud_{A_{\mathfrak{p}}}(B_{\mathfrak{p}})\leq
  \ud_{A}(B)$, for every $\mfp\in\Spec(A)$.
\item[$(f)$] $\ud_K(L)\leq [L:K]$.
\end{itemize}
\end{proposition}
\begin{proof} $(a)$, $(b)$ and $(c)$ can be found in 
\cite[Corollary~6.3, Corollary~6.2 and Proposition~6.8]{gp}. For
$(d)$, since $K=S^{-1}A\subset S^{-1}B$ is an integral ring extension,
then $S^{-1}B$ is a zero-dimensional domain, hence a field, lying
inside $L=Q(B)$. Thus $S^{-1}B=L$. Let us prove $(e)$. Take
$\mfp\in\Spec(A)$, so $A\setminus \mfp\subseteq S$. Since $A\subset B$
is an integral ring extension, then $A_{\mathfrak{p}}\subset
B_{\mathfrak{p}}$ and
\begin{eqnarray*}
K=S^{-1}A=(A_{\mathfrak{p}}\setminus\{0\})^{-1}A_{\mathfrak{p}}\subset
(A_{\mathfrak{p}}\setminus\{0\})^{-1}B_{\mathfrak{p}}=S^{-1}B=L
\end{eqnarray*}
are integral ring extensions. Applying $(c)$ twice, we get
$(e)$. Finally, applying $(d)$ and $(a)$, one has
$\ud_K(L)=\ud_{S^{-1}A}(S^{-1}B))\leq
\mu_{S^{-1}A}(S^{-1}B)=\mu_K(L)=[L:K]$, which proves $(f)$.
\end{proof}

\begin{notation}\label{diagram}
The following picture can help in reading the paper.
\begin{eqnarray*}
\begin{matrix}
\ud_K(L) & \leq & \ud_A(B)\\ 
\mbox{\rule{0.45pt}{2.25mm}}\, \wedge && 
\mbox{\rule{0.45pt}{2.25mm}}\, \wedge \\
[L:K] & \leq & \mu_A(B).
\end{matrix}
\end{eqnarray*}
We say that $A\subset B$ has {\em minimal integral degree} when
$\ud_K(L)=\ud_A(B)$. Similarly, we say that $A\subset B$ has {\em
  maximal integral degree} when $\ud_A(B)=\mu_A(B)$.
\end{notation}

\begin{remark}\label{max-int-deg}
By Proposition~\ref{first-prop},~$(b)$, $A\subset B$ simple implies
$A\subset B$ has maximal integral degree. We will see that the
converse is true for finite field extensions (see
Proposition~\ref{field-separable-case}). However, in general,
$A\subset B$ of maximal integral degree does not imply $A\subset B$
simple. Take for instance $A=\mbz$ and $B$ the ring of integers of an
algebraic number field $L$, i.e., $B$ is the integral closure of
$A=\mbz$ in $L$, a finite field extension of the field of rational
numbers $K=\mbq$. Then $\ud_A(B)=\mu_A(B)$ (see
Corollary~\ref{dedekind}). Nevertheless, not every ring of integers
$B$ is a simple extension of $A=\mbz$. We will take advantage of this
fact in Example~\ref{no-sm}.
\end{remark}

Clearly, there are integral ring extensions of non-maximal integral
degree. This can already happen with affine domains, as shown in the
next example.

\begin{example}\label{d<mu}
Let $k$ be a field and $t$ a variable over $k$. Let
$A=k[t^3,t^8,t^{10}]$ and $B=k[t^3,t^4,t^5]$. Then $A\subset B$ is a
finite ring extension with $\ud_A(B)=2$ and $\mu_A(B)=3$.
\end{example}
\begin{proof}
Since $k[t^3]\subset A$, then $B=\langle 1,t^4,t^5\rangle_A$ and
$A\subset B$ is a finite ring extension with $\mu_A(B)\leq 3$. If
$x=a+bt^4+ct^5\in B$, with $a,b,c\in A$, then $x^2-2ax\in
A$. Therefore $\ud_A(B)=2$. Let us see that $\mu_A(B)=3$.  Suppose
that there exist $f,g\in B$ such that $B=\langle f,g\rangle_A$, i.e.,
$1,t^4,t^5\in \langle f,g\rangle_A$. Write $f=a_0+t^3f_1$ and
$g=b_0+t^3g_1$, with $a_0,b_0\in k$ and $f_1,g_1\in k[t]$. Since
$1\in\langle f,g\rangle_{A}$, one can suppose that $a_0=1$ and
$b_0=0$. Thus $f=1+a_3t^3+a_4t^4+a_5t^5+\ldots $ and
$g=b_3t^3+b_4t^4+b_5t^5+\ldots $. In particular, every element of
$\langle f,g\rangle_A$ is of the form:
\begin{eqnarray*}
&&(\lambda_0+\lambda_3t^3+\lambda_6t^6+\lambda_8t^8+\ldots )
  (1+a_3t^3+a_4t^4+a_5t^5+\ldots)+\\&&
  (\mu_0+\mu_3t^3+\mu_6t^6+\mu_8t^8+\ldots )
  (b_3t^3+b_4t^4+b_5t^5+\ldots)=\\&&
  (\lambda_0)+(\lambda_3+\lambda_0a_3+\mu_0b_3)t^3+(\lambda_0a_4+\mu_0b_4)t^4+
  (\lambda_0a_5+\mu_0b_5)t^5+\ldots.
\end{eqnarray*}
From $t^4\in \langle f,g\rangle_A$, one deduces that ($\lambda_0=0$
and) $b_4\neq 0$. Hence, one can suppose that $a_4=0$. From $1\in
\langle f,g\rangle_A$, it follows that ($\lambda_0=1$, $\mu_0=0$ and)
$a_5=0$. From $t^4\in \langle f,g\rangle_A$ again, now it follows
$b_5=0$. But from $t^5\in \langle f,g\rangle_A$, one must have
$b_5\neq 0$, a contradiction. Hence $\mu_A(B)=3$.
\end{proof}

\begin{remark}\label{min-int-deg}
In the example above $K=L$ and so $A\subset B$ does not have minimal
integral degree. We will prove that if $A\subset B$ is projective and
finite with $K\subset L$ simple, or if $A$ is integrally closed, then
$A\subset B$ has minimal integral degree (see
Theorem~\ref{proj-klsimple} and Proposition~\ref{kronecker}).
\end{remark}

As for the finiteness of the integral degree, we recall the following.

\begin{remark}\label{d-finite-mu-infinite}
There exist one-dimensional noetherian local domains $A$ with integral
closure $\overline{A}$ such that $\ud_A(\overline{A})$ is finite while
$\mu_A(\overline{A})$ is infinite (see
\cite[Proposition~6.5]{gp}). There exist one-dimensional noetherian
domains $A$ such that $\ud_A(\overline{A})$ is infinite (see
\cite[Example~6.6]{gp}).
\end{remark}

Next we prove that the integral degree coincides with the supremum of
the integral degrees of the localizations. (The analogue for
$\mu_A(B)$ is not true in general.)

\begin{proposition}\label{local-invariant}
Let $A\subset B$ be an integral ring extension. For any $b\in B$,
there exists a maximal ideal $\mfm\in\Max(A)$ such that
$\id_A(b)=\id_{A_{\mathfrak{m}}}(b/1)$. In particular,
\begin{eqnarray*}
\ud_A(B)=\sup\{\ud_{A_{\mathfrak{m}}}(B_{\mathfrak{m}})\mid
\mfm\in\Max(A)\}=\sup\{\ud_{A_{\mathfrak{p}}}(B_{\mathfrak{p}})\mid
\mfp\in\Spec(A)\}.
\end{eqnarray*}
Furthermore, if $\ud_A(B)$ is finite, then there exists
$\mfm\in\Max(A)$ such that
$\ud_A(B)=\ud_{A_{\mathfrak{m}}}(B_{\mathfrak{m}})$.
\end{proposition}
\begin{proof}
If $\id_A(b)=1$, then $b\in A$. Thus, for any $\mfm\in\Max(A)$,
$b/1\in A_{\mathfrak{m}}$, so $\id_{A_{\mathfrak{m}}}(b/1)=1$ and
$\id_A(b)=\id_{A_{\mathfrak{m}}}(b/1)$. Suppose that $\id_A(b)=n\geq
2$. Then
\begin{eqnarray*}
A[b]/\langle 1,b,\ldots ,b^{n-2}\rangle\neq 0\espai\mbox{and}\espai
A[b]/\langle 1,b,\ldots ,b^{n-1}\rangle=0.
\end{eqnarray*} 
Clearly, for every $\mfp\in\Spec(A)$, and for every $m\geq 1$,
\begin{eqnarray*}
(A[b]/\langle 1,b,\ldots
,b^{m}\rangle)_{\mathfrak{p}}=A_{\mathfrak{p}}[b/1]/\langle
1,b/1,\ldots ,b^{m}/1\rangle. 
\end{eqnarray*}
In particular, $A_{\mathfrak{p}}[b/1]/\langle 1,b/1,\ldots
,b^{n-1}/1\rangle=0$, for every $\mfp\in\Spec(A)$. Since $A[b]/\langle
1,b,\ldots ,b^{n-2}\rangle\neq 0$, then there exists a maximal ideal
$\mfm\in\Max(A)$ with
\begin{eqnarray*}
A_{\mathfrak{m}}[b/1]/\langle 1,b/1,\ldots ,b^{n-2}/1\rangle\neq
0\espai\mbox{and}\espai A_{\mathfrak{m}}[b/1]/\langle 1,b/1,\ldots
,b^{n-1}/1\rangle=0.
\end{eqnarray*}
Therefore, $\id_{A_{\mathfrak{m}}}(b/1)=n$ and
$\id_A(b)=\id_{A_{\mathfrak{m}}}(b/1)$. In particular, 
\begin{eqnarray*}
\ud_A(B)\leq \sup\{\ud_{A_{\mathfrak{m}}}(B_{\mathfrak{m}})\mid
\mfm\in\Max(A)\}\leq \sup\{\ud_{A_{\mathfrak{p}}}(B_{\mathfrak{p}})\mid
\mfp\in\Spec(A)\}.
\end{eqnarray*}
On the other hand, by Proposition~\ref{first-prop},
$\sup\{\ud_{A_{\mathfrak{p}}}(B_{\mathfrak{p}})\mid
\mfp\in\Spec(A)\}\leq \ud_A(B)$. Finally, if $\ud_A(B)$ is finite,
then there exists $b\in B$ such that $\id_A(b)=\ud_A(B)$. We have just
shown above that there exists a maximal ideal $\mfm\in\Max(A)$ with
$\id_A(b)=\id_{A_{\mathfrak{m}}}(b/1)$. Therefore
\begin{eqnarray*}
\ud_A(B)=\id_A(b)=\id_{A_{\mathfrak{m}}}(b/1)\leq
\ud_{A_{\mathfrak{m}}}(B_{\mathfrak{m}})\leq\ud_A(B)
\end{eqnarray*}
and the equality holds.
\end{proof}

\begin{remark}\label{generic-flatness}
Suppose that $A\subset B$ is finite. Since $A$ is a domain, by generic
flatness, there exists $f\in A\setminus\{0\}$ such that $A_f\subset
B_f$ is a finite free extension (see, e.g.,
\cite[Theorem~22.A]{matsumura-ca}). In particular, for every $\mfp\in
D(f)=\Spec(A)\setminus V(f)$, $A_{\mathfrak{p}}\subset
B_{\mathfrak{p}}$ is a finite free ring extension. So
$\ud_A(B)=\max\{\ud_1,\ud_2\}$, where
$\ud_1=\sup\{\ud_{A_{\mathfrak{p}}}(B_{\mathfrak{p}})\mid \mfp\in
V(f)\}$ and $\ud_2=\sup\{\ud_{A_{\mathfrak{p}}}(B_{\mathfrak{p}})\mid
A_{\mathfrak{p}}\subset B_{\mathfrak{p}}\mbox{ is free}\}$.
Therefore, if one is able to control the integral degree for finite
free ring extensions, the calculation of $\ud_A(B)$ is reduced to find
$\ud_1=\sup\{\ud_{A_{\mathfrak{p}}}(B_{\mathfrak{p}})\mid \mfp\in
V(f)\}$, where $V(f)$ is a proper closed set of $\Spec(A)$. We will
come back to this question in Theorem~\ref{free-klsimple}.
\end{remark}

\section{Sub-multiplicativity}\label{sm}

In this section we study the sub-multiplicativity of the integral
degree with respect to $A\subset B$, i.e., whether $\ud_A(C)\leq
\ud_A(B)\ud_B(C)$ holds for every integral ring extension $B\subset
C$. Observe that, in this situation, $A\subset C$ is an integral ring
extension too and, by definition, $\ud_B(C)\leq \ud_A(C)$. We start
with a useful criterion to determine possible bounds $\nu\,\in\mbn$ in
the inequality $\ud_A(C)\leq \nu\,\ud_B(C)$.

\begin{lemma}\label{equivalences} 
Let $A\subset B$ and $B\subset C$ be two integral ring extensions.
Set $\nu\in\mbn$. The following conditions are equivalent.
\begin{itemize}
\item[$(i)$] $\ud_A(D)\leq \nu\,\ud_B(D)$, for every ring $D$ such that
  $B\subseteq D\subseteq C$;
\item[$(ii)$] $\ud_A(D)\leq \nu\,\ud_B(D)$, for every ring $D$ such that
  $D=B[\alpha]$ for some $\alpha\in C$;
\item[$(iii)$] $\id_A(\alpha)\leq \nu\,\id_B(\alpha)$, for every
element $\alpha\in C$.
\end{itemize}
In particular, if $(iii)$ holds, then $\ud_A(C)\leq \nu\,\ud_B(C)$.
\end{lemma}
\begin{proof}
Clearly, $(i)\Rightarrow (ii)$. Let $\alpha\in C$; in particular,
$\alpha$ is integral over $A$. Since $A[\alpha]\subset B[\alpha]$,
then $\id_A(\alpha)=\ud_A(A[\alpha])\leq\ud_A(B[\alpha])$. By
hypothesis $(ii)$, $\ud_A(B[\alpha])\leq
\nu\,\ud_B(B[\alpha])=\nu\,\id_B(\alpha)$. Therefore, $\id_A(\alpha)\leq
\nu\,\id_B(\alpha)$, which proves $(ii)\Rightarrow (iii)$. To see
$(iii)\Rightarrow (i)$, take $D$ with $B\subseteq D\subseteq C$ and
$\alpha\in D$, which will be integral over $B$ and, hence, integral
over $A$. By hypothesis $(iii)$, $\id_A(\alpha)\leq
\nu\,\id_B(\alpha)\leq \nu\,\ud_B(D)$. Taking supremum over all $\alpha\in
D$, $\ud_A(D)\leq \nu\,\ud_B(D)$. 

Finally, if $(iii)$ holds, then $(i)$ holds for $D=C$, so
$\ud_A(C)\leq\nu\,\ud_B(C)$.
\end{proof}

The next result shows that we can take $\nu=\mu_A(B)$ as a particular
$\nu\in\mbn$, understanding that if $A\subset B$ is not finite, then
$\mu_A(B)=\infty$ and the inequality is trivial.

\begin{theorem}\label{m-equal-mu}
Let $A\subset B$ and $B\subset C$ be two integral ring extensions.
Then, for every $\alpha\in C$,
\begin{eqnarray*}
\id_A(\alpha)\leq \mu_A(B)\id_B(\alpha).
\end{eqnarray*}
In particular, 
\begin{eqnarray*}
\ud_A(C)\leq \mu_A(B)\ud_B(C).
\end{eqnarray*}
\end{theorem}
\begin{proof}
Let $\alpha\in C$, which is integral over $B$ and $A$. Then
\begin{eqnarray*}
\id_A(\alpha)=\ud_A(A[\alpha])\leq \ud_A(B[\alpha])\leq
\mu_A(B[\alpha])\leq
\mu_A(B)\mu_B(B[\alpha])=\mu_A(B)\id_B(\alpha).
\end{eqnarray*}
To finish, apply Lemma~\ref{equivalences}.
\end{proof}

\begin{remark}\label{det-trick}
A proof of Theorem~\ref{m-equal-mu} using the standard ``determinantal
trick'' would be as follows. Suppose $B=\langle b_1,\ldots
,b_n\rangle_A$, with $\mu_A(B)=n$, and consider $\alpha\in C$ with
$\id_B(\alpha)=m$. Let $X$ be the following $nm\times 1$ vector, whose
entries form an $A$-module generating set of $B[\alpha]$,
\begin{eqnarray*}
X^{\top}=(1,\alpha,\ldots,\alpha^{m-1},b_1,b_1\alpha,\ldots,
b_1\alpha^{m-1},\ldots, b_n,b_n\alpha,\ldots, b_n\alpha^{m-1}).
\end{eqnarray*}
Then there exists a $nm$ square matrix $P$ with coefficients in $A$,
such that $\alpha X=PX$. Therefore, $(\alpha\mbox{\rm
  I}-P)X=0$. Multiplying by the adjugate matrix (that is, the
transpose of the cofactor matrix) leads to
$Q_P(\alpha)=\det(\alpha\mbox{\rm I}-P)=0$, where $Q_P(T)$ is the
characteristic polynomial of $P$ (recall that $C$ is a domain). Hence
$\id_A(\alpha)\leq\deg Q_P(T)=nm=\mu_A(B)\id_B(\alpha)$.
\end{remark}

As an immediate consequence of Theorem~\ref{m-equal-mu}, we obtain the
sub-multiplicativity of the integral degree with respect to integral
ring extensions of maximal integral degree.

\begin{corollary}\label{sm-mid}
Let $A\subset B$ and $B\subset C$ be two integral ring extensions.  If
$\ud_A(B)=\mu_A(B)$, then
\begin{eqnarray*}
\ud_A(C)\leq \ud_A(B)\ud_B(C).
\end{eqnarray*}
\end{corollary}

To finish this section we recover part of a result shown in \cite{gp},
but now with a slightly different proof.

\begin{corollary}\label{exponential}
{\rm (\cite[Proposition~6.7]{gp})} Let $A\subset B$ and $B\subset C$
be two integral ring extensions. Then, for every $\alpha\in C$,
\begin{eqnarray*}
\id_{A}(\alpha)\leq \ud_{A}(B)^{{\rm d}_{B}(C)}\id_{B}(\alpha).
\end{eqnarray*}
In particular,
\begin{eqnarray*}
\ud_{A}(C)\leq \ud_{A}(B)^{{\rm d}_{B}(C)}\ud_{B}(C).
\end{eqnarray*}
Furthermore, if $A\subset B$ and $B\subset C$ have finite integral
degrees, then $A\subset C$ has finite integral degree.
\end{corollary}
\begin{proof}
Let $\alpha\in C$; in particular, $\alpha$ is integral over $B$ and
over $A$. Let $\mdp(T)$ be a minimal degree polynomial of $\alpha$ over
$B$, $\mdp(T)=T^{n}+b_1T^{n-1}+\ldots +b_n\in B[T]$, so that
$n=\id_B(\alpha)\leq \ud_B(C)$.  Set $E=A[b_1,\ldots ,b_n]$, where
$A\subseteq E\subseteq B$. Therefore,
\begin{eqnarray*}
\id_A(\alpha)=\ud_A(A[\alpha])\leq \ud_A(E[\alpha])\leq
\mu_A(E[\alpha])\leq \mu_A(E)\mu_E(E[\alpha]),
\end{eqnarray*}
where clearly $\mu_A(E)\leq \prod_{i=1}^n\id_A(b_i)\leq\ud_A(B)^{{\rm
    d}_B(C)}$, and $\mu_E(E[\alpha])=\id_E(\alpha)=\id_B(\alpha)$. To
finish, apply Lemma~\ref{equivalences}.
\end{proof}

\section{Integral degree of algebraic field 
extensions}\label{field}

In this section, we suppose that $A=K$, $B=L$ and $C=M$ are fields.
For ease of reading, we begin by recalling some definitions and basic
facts (see, e.g., \cite[Chapter~V]{bourbaki} and
\cite{jacobson-book}).

\begin{reminder}\label{decomposition}
Let $K\subset L$ be a finite field extension.
\begin{itemize}
\item[$\bullet$] A polynomial is separable if it has no multiple
  roots (in any field extension). The extension $K\subset L$ is
  separable if every element of $L$ is the root of a separable
  polynomial of $K[T]$. A field $K$ is perfect if either has
  characteristic zero or else, when it has characteristic $p>0$, every
  element is a $p$-th power in $K$. If $K$ is perfect, then
  $K\subset L$ is separable.
\item[$\bullet$] The primitive element theorem states that a finite
  separable field extension is simple. Even more, there exists an
  ``extended'' version which affirms that a simple algebraic field
  extension of a finite separable field extension is again simple (cf.
  \cite[III, Chapter~I, \S~11, Theorem~14]{jacobson-book}).
\item[$\bullet$] Let $K_s$ be the separable closure of $K$ in $L$,
  i.e., the set of all elements of $L$ which are separable over
  $K$. Then $K_s$ is a field and $K\subset K_s$ is a separable
  extension. Its degree $[K_s:K]$ is called the separable degree and
  is denoted by $[L:K]_s:=[K_s:K]$.
\end{itemize}
For the rest of the reminder, suppose that $K$ has characteristic
$p>0$ and let $K_s$ be the separable closure of $K$ in $L$.
\begin{itemize}
\item[$\bullet$] Then $K_s\subset L$ is a purely inseparable field
  extension, i.e., for every element $\alpha\in L$, there exists an
  integer $m\geq 1$ such that $\alpha^{p^m}\in K_s$. The least such
  integer $m$ is called the height of $\alpha$ over $K_s$. Let
  $\height_{K_s}(\alpha)$ stand for the height of $\alpha$ over
  $K_s$. Set $h=\sup\{\height_{K_s}(\alpha)\mid \alpha\in L\}$ and
  call $h$ the {\em height of the purely inseparable extension}
  $K_s\subset L$.
\item[$\bullet$] Given $\alpha\in L$ with $\height_{K_s}(\alpha)=m$,
  setting $a=\alpha^{p^m}$, one proves that $T^{p^m}-a$ is irreducible
  in $K_s[T]$ (see, e.g., \cite[Chapter~V, \S~5]{bourbaki}). Thus
  $[K_s(\alpha):K_s]=p^m$. Since $K_s\subset L$ is a finite extension,
  then $L=K_s(\alpha_1,\ldots,\alpha_r)$, where each $\alpha_i$ is
  purely inseparable over $K_s(\alpha_1,\ldots ,\alpha_{i-1})$,
  $i=2,\ldots ,r$. Hence $[L:K_s]=p^e$, for some $e\geq 1$. Call $e$
  the {\em exponent of the purely inseparable extension} $K_s\subset
  L$. Note that, since $[K_s(\alpha):K_s]$ (which is $p^m$) divides
  $[L:K_s]=[L:K_s(\alpha)][K_s(\alpha):K_s]$ (which is $p^e$), then
  $m\leq e$ and $h\leq e$.
\end{itemize}
\end{reminder}

Our first result characterizes simple finite field extensions
as finite field extensions of maximal integral degree.

\begin{proposition}\label{field-separable-case}
Let $K\subset L$ be a finite field extension. Then $K\subset L$ is
simple if and only if $\ud_K(L)=[L:K]$.
\end{proposition}
\begin{proof}
Since $K\subset L$ is an algebraic extension, $K(\alpha)=K[\alpha]$,
for any $\alpha\in L$. By Proposition~\ref{first-prop},~$(b)$,
$\id_K(\alpha)=\ud_K(K(\alpha))=[K(\alpha):K]$. Therefore,
\begin{eqnarray}\label{chain}
[L:K]=[L:K(\alpha)][K(\alpha):K]=[L:K(\alpha)]\ud_K(K(\alpha))=
[L:K(\alpha)]\id_K(\alpha).
\end{eqnarray}
If $K\subset L$ is simple, then $L=K(\alpha)$, for some $\alpha\in
L$. Using \eqref{chain}, it follows that
\begin{eqnarray*}
[L:K]=[L:K(\alpha)]\ud_K(K(\alpha))=\ud_K(L).
\end{eqnarray*}
Conversely, if $\ud_K(L)=[L:K]<\infty$, by definition, there exists
$\alpha\in L$ with $\id_K(\alpha)=[L:K]$. By \eqref{chain} again, it
follows that $[L:K(\alpha)]=1$ and $K\subset L$ is simple.
\end{proof}

Using the primitive element theorem we obtain the following result.

\begin{corollary}\label{sep=}
Let $K\subset L$ be a finite separable field extension. Then
$\ud_K(L)=[L:K]$.
\end{corollary}

The ``extended'' version of the primitive element theorem will show to
be very useful in proving the next result.

\begin{proposition}\label{height}
Let $K\subset L$ be a finite field extension. Suppose that $K$ has
characteristic $p>0$ and let $K_s$ be the separable closure of $K$ in
$L$. Set $h$ and $e$ to be the height and exponent, respectively, of
the finite purely inseparable field extension $K_s\subset L$. Then the
following hold.
\begin{itemize}
\item[$(a)$] $\ud_K(L)=\ud_K(K_s)\ud_{K_s}(L)$.
\item[$(b)$] For every $\alpha\in L$, $\ud_{K_s}(\alpha)=p^m$, where
  $m$ is the height of $\alpha$ over $K_s$.
\item[$(c)$] $\ud_{K_s}(L)=p^h$ and $[L:K_s]=p^e$.
\item[$(d)$] $[L:K]_s$ divides $\ud_K(L)$ and $\ud_K(L)$ divides
  $[L:K]$. Concretely, 
\begin{eqnarray*}
\ud_K(L)=[L:K]_sp^h\mbox{ and }[L:K]=p^{e-h}\ud_K(L).
\end{eqnarray*}
\end{itemize}
\end{proposition}
\begin{proof}
By Corollary~\ref{sep=}, $\ud_{K}(K_s)=[K_s:K]$. By
Corollary~\ref{sm-mid}, $\ud_K(L)\leq \ud_K(K_s)\ud_{K_s}(L)$. To see
the other inequality, take $\alpha\in L$ with
$\id_{K_s}(\alpha)=\ud_{K_s}(L)$. By
Proposition~\ref{first-prop},~$(b)$,
\begin{eqnarray*}
\id_{K_s}(\alpha)=\ud_{K_s}(K_s[\alpha])=\mu_{K_s}(K_s[\alpha]).
\end{eqnarray*}
By the extended primitive element theorem, $K\subset K_s[\alpha]$ is a
simple algebraic field extension (cf. \cite[III, Chapter~I, \S~11,
  Theorem~14]{jacobson-book}). Hence, by
Proposition~\ref{first-prop},~$(b)$,
$\ud_{K}(K_s[\alpha])=[K_s[\alpha]:K]$. Since $K\subset K_s$ is a
finite separable extension, by Corollary~\ref{sep=},
$\ud_{K}(K_s)=[K_s:K]$. Writing all together:
\begin{eqnarray*}
\ud_K(K_s)\ud_{K_s}(L)=[K_s:K]\id_{K_s}(\alpha)=
[K_s:K]\mu_{K_s}(K_s[\alpha])=[K_s[\alpha]:K]=
\ud_K(K_s[\alpha])\leq \ud_K(L).
\end{eqnarray*}
This proves $(a)$. Let $\alpha\in L$ with
$\height_{K_s}(\alpha)=m$. Set $a=\alpha^{p^m}$. Then $T^{p^m}-a\in
K_s[T]$ is irreducible in $K_s[T]$ and hence it is the minimimal
polynomial of $\alpha$ over $K_s$. It follows that
$\id_{K_s}(\alpha)=p^m$. This proves $(b)$. Therefore,
\begin{eqnarray*}
\ud_{K_s}(L)=\sup\{\id_{K_s}(\alpha)\mid \alpha\in L\}= \sup\{p^{{\rm
    ht}_{K_s}(\alpha)}\mid \alpha\in L\}= p^{\sup\{{\rm
    ht}_{K_s}(\alpha)\mid \alpha\in L\}}=p^h,
\end{eqnarray*}
which proves $(c)$. Finally, $(d)$ follows from $(a)$, $(c)$ and
Corollary~\ref{sep=} applied repeatedly. Indeed,
\begin{eqnarray*}
\ud_K(L)=\ud_K(K_s)\ud_{K_s}(L)=[K_s:K]p^h=[L:K]_sp^h
\end{eqnarray*}
and
\begin{eqnarray*}
[L:K]=[L:K_s][K_s:K]=p^e\ud_{K}(K_s)=p^{e-h}\ud_{K_s}(L)\ud_{K}(K_s)
=p^{e-h}\ud_{K}(L).
\end{eqnarray*}
\end{proof}

Here there is an example of a finite field extension of non maximal
integral degree. 

\begin{example}\label{ud-not-equal-mu}
Let $p>1$ be a prime and $K=\mbf_p(u_1^p,u_2^p)$, where $u_1,u_2$ are
algebraically independent over $\mbf_p$. Set $L=K[u_1,u_2]$. Then
$K\subset L$ is a finite purely inseparable field extension with
$\ud_K(L)=p$. However $[L:K]=p^2$.
\end{example}
\begin{proof}
Any $\beta\in L$ is of the form $\beta=\sum_{0\leq i,j\leq
  p-1}a_{i,j}u_1^{i}u_2^{j}$, with $a_{i,j}\in K$. So
\begin{eqnarray*}
\beta^p=\sum_{0\leq i,j\leq p-1}a_{i,j}^pu_1^{ip}u_2^{jp}=
\sum_{0\leq i,j\leq p-1}a_{i,j}^p(u_1^p)^i(u_2^p)^{j},
\end{eqnarray*}
which is an element of $K$. Therefore $\beta^p\in K$ and
$\id_K(\beta)\leq p$. Since $\id_K(u_1)=p$, it follows that
$\ud_K(L)=p$. Since $K\varsubsetneq K(u_1)\varsubsetneq L$ are finite
field extensions, each one of degree $p$, by the multiplicative
formula for algebraic field extensions,
$[L:K]=[L:K(u_1)][K(u_1):K]=p^2$.
\end{proof}

Similarly, we obtain an example of an infinite field extension with
finite integral degree (see also Remark~\ref{d-finite-mu-infinite}).

\begin{example}\label{d-finite-mu-infinite-fields}
Let $p>1$ be a prime and $K=\mbf_p(u_1^p,u_2^p,\ldots )$, where
$u_1,u_2,\ldots $ are algebraically independent over $\mbf_p$. Set
$L=K[u_1,u_2,\ldots ]$. Then $\ud_K(L)=p$ but $[L:K]=\infty$.
\end{example}

Now we prove the sub-multiplicativity of the integral degree with
respect to an algebraic field extension $K\subset L$.

\begin{theorem}\label{sm-afe}
Let $K\subset L$ and $L\subset M$ be two algebraic field extensions.
Then, for every $\alpha\in M$,
\begin{eqnarray*}
\id_K(\alpha)\leq \ud_K(L)\id_L(\alpha).
\end{eqnarray*}
In particular,
\begin{eqnarray*}
\ud_{K}(M)\leq\ud_{K}(L)\ud_{L}(M).
\end{eqnarray*}
\end{theorem}
\begin{proof}
We can assume that $\ud_K(L)$ and $\ud_L(M)$ are finite. 

Let $\alpha\in M$ and let $\mdp(T)=T^n+b_1T^{n-1}+\ldots +b_n\in L[T]$
be a minimal degree polynomial of $\alpha$ over $L$. Let $h$ be the
height of the purely inseparable field extension $K_s\subset L$, where
$K_s$ is the separable closure of $K$ in $L$.  Let $p=\charac(K)$.  If
$K$ has characteristic $0$, we understand that $p^h=1$. Then
$0=\mdp(\alpha)^{p^h} =\alpha^{np^h}+b_1^{p^h}\alpha^{(n-1)p^h}+\ldots
+b_n^{p^h}$. It follows that $\alpha$ is a root of a monic polynomial
in $K_s[T]$ of degree $p^h\id_L(\alpha)$. So we have
\begin{eqnarray*}
&\id_K(\alpha)=&\ud_K(K[\alpha])\leq \ud_K(K_s[\alpha])\leq
  \mu_K(K_s[\alpha])\leq \mu_K(K_s)\cdot\mu_{K_s}(K_s[\alpha])=\\&&
  =\ud_K(K_s)\cdot\id_{K_s}(\alpha) \leq
  \ud_K(K_s)p^h\id_L(\alpha)=\ud_K(L)\id_L(\alpha).
\end{eqnarray*}
To finish, apply Lemma~\ref{equivalences}.
\end{proof}

Though sub-multiplicative, the integral degree might not be
multiplicative, even for two simple algebraic field extensions.

\begin{example}\label{strict-inequality}
Let $p>1$ be a prime and let $K=\mbf_p(u^p_1,u^p_2)$, where $u_1,u_2$
are algebraically independent over $\mbf_p$. Set $L=K[u_1]$ and
$M=L[u_2]$. Then $K\subset L$ and $L\subset M$ are two finite field
extensions with $\ud_K(M)=p$ and
$\ud_{K}(L)\ud_{L}(M)=\id_K(u_1)\id_L(u_2)=p^2$ (see
Example~\ref{ud-not-equal-mu} and Proposition~\ref{first-prop}).
\end{example}

However, for finite separable field extensions, multiplicativity
holds.

\begin{remark}\label{mult-fields}
Let $K\subset L$ be a finite separable field extension and $L\subset
M$ be a simple algebraic field extension. Then
\begin{eqnarray*}
\ud_{K}(M)=\ud_{K}(L)\ud_{L}(M).
\end{eqnarray*}
\end{remark}
\begin{proof}
By the extended primitive element theorem, $K\subset M$ is simple.
Hence, by Proposition~\ref{field-separable-case}, $\ud_{K}(M)=[M:K]$,
$\ud_{K}(L)=[L:K]$ and $\ud_{L}(M)=[M:L]$.
\end{proof}

\section{Integral degree of projective finite ring 
extensions}\label{proj}

We return to the general hypotheses: $A\subset B$ and $B\subset C$ are
integral ring extensions of integral domains, and $K$, $L$ and $M$ are
their fields of fractions, respectively. In this section we are
interested in the integral degree of projective finite ring extensions
(by a projective, respectively free, ring extension $A\subset B$ we
understand that $B$ is a projective, respectively free, $A$-module).
We begin by recalling some well-known definitions and facts (see,
e.g., \cite[Chapter~IV, \S~2, 3]{kunz}).

\begin{reminder}\label{def-rank}
Let $A$ be a domain and let $N$ be a finitely generated $A$-module.
\begin{enumerate}
\item[$\bullet$] $N$ is a free $A$-module if it has a basis, i.e., a
  linearly independent system of generators. The {\em rank of a free
    module $N$}, $\rank_A(N)$, is defined as the cardinality of
  (indeed, any) a basis. Clearly, $N$ is free of rank $n$ if and only
  if $N\cong A^n$. If $N$ is a free $A$-module, the minimal generating
  sets are just the bases of $N$. In particular,
  $\mu_A(N)=\rank_A(N)$.
\item[$\bullet$] $N$ is a projective $A$-module if there exists an
  $A$-module $N^{\prime}$ such that $N\oplus N^{\prime}$ is free. One
  has that $N$ is projective if and only if $N$ is finitely
  presentable and locally free. The {\em rank of a projective module
    $N$ at a prime $\mfp$}, $\rank_{\mathfrak{p}}(N)$, is defined as
  the rank of the free $A_{\mathfrak{p}}$-module $N_{\mathfrak{p}}$,
  i.e.,
  $\rank_{\mathfrak{p}}(N)=\rank_{A_{\mathfrak{p}}}(N_{\mathfrak{p}})
  =\mu_{A_{\mathfrak{p}}}(N_{\mathfrak{p}})=\dim_{k(\mathfrak{p})}(N\otimes
  k(\mfp))$, where $k(\mfp)=A_{\mathfrak{p}}/\mfp A_{\mathfrak{p}}$
  stands for the residue field of $A$ at $\mfp$.
\item[$\bullet$] If $N$ is projective, then
  $\mfp\mapsto\rank_{\mathfrak{p}}(N)$ is constant (since $A$ is a
  domain, $\Spec(A)$ is connected) and is simply denoted by
  $\rank_A(N)$. In particular, on taking the prime ideal $(0)$, then
  $\rank_A(N)=\mu_K(N\otimes K)=\rank_{\mathfrak{p}}(N)$, for every
  prime ideal $\mfp$. Clearly, when $N$ is free both definitions of
  rank coincide.
\end{enumerate}
\end{reminder}

\begin{theorem}\label{proj-klsimple}
Let $A\subset B$ be a projective finite ring extension. Then
\begin{eqnarray*}
\ud_K(L)\leq \ud_A(B)\leq \rank_A(B)=[L:K].
\end{eqnarray*}
If moreover $K\subset L$ is simple, then
\begin{eqnarray*}
\ud_K(L)=\ud_A(B)=\rank_A(B)=[L:K].
\end{eqnarray*}
\end{theorem}
\begin{proof}
By Proposition~\ref{local-invariant}, there exists a maximal ideal
$\mfm$ of $A$ such that
$\ud_A(B)=\ud_{A_\mathfrak{m}}(B_{\mathfrak{m}})$. By
Proposition~\ref{first-prop} and using that $B_{\mathfrak{m}}$ is
$A_{\mathfrak{m}}$-free and $B$ is $A$-projective, then
\begin{eqnarray*}
\ud_K(L)\leq \ud_A(B)=\ud_{A_\mathfrak{m}}(B_{\mathfrak{m}})\leq
\mu_{A_\mathfrak{m}}(B_{\mathfrak{m}})=\rank_{A_\mathfrak{m}}(B_{\mathfrak{m}})
=\rank_A(B)=\mu_K(B\otimes K)=[L:K].
\end{eqnarray*}
To finish, recall that if $K\subset L$ is simple, then
$\ud_K(L)=[L:K]$ (see Proposition~\ref{first-prop} or
~\ref{field-separable-case}).
\end{proof}

The next result characterizes finite ring extensions of maximal and
minimal integral degree at the same time.

\begin{corollary}\label{free-klsimple}
Let $A\subset B$ be a finite ring extension. 
\begin{itemize}
\item[$(a)$] $A\subset B$ is free if and only if $[L:K]=\mu_A(B)$.
\item[$(b)$] $A\subset B$ is free and $K\subset L$ is simple if and
  only if $\ud_K(L)=\ud_A(B)=\mu_A(B)$.
\end{itemize}
\end{corollary}
\begin{proof}
If $A\subset B$ is free, then $\mu_A(B)=\rank_A(B)=[L:K]$ (see
Reminder~\ref{def-rank} and
Theorem~\ref{proj-klsimple}). Reciprocally, suppose that
$[L:K]=\mu_A(B)$. Set $\mu_A(B)=n$ and let $u_1,\ldots ,u_n$ be a
system of generators of the $A$-module $B$. Thus, $u_1,\ldots ,u_n$ is
a system of generators of the $K$-module $L$, where $n=[L:K]$ (recall
that, if $S=A\setminus\{0\}$, then $K=S^{-1}A$ and $L=S^{-1}B$, cf.
Proposition~\ref{first-prop}). Hence, they are a $K$-basis of $L$, so
$K$-linearly independent. In particular, $u_1,\ldots ,u_n$ are
$A$-linearly independent. Since they also generate $B$, one concludes
that $u_1,\ldots ,u_n$ is an $A$-basis of $B$ and that $B$ is a free
$A$-module. This shows $(a)$.  Since $\ud_K(L)\leq
\ud_A(B)\leq\mu_A(B)$ and $\ud_K(L)\leq [L:K]\leq\mu_A(B)$ (see
Notation~\ref{diagram}), part $(b)$ follows from part $(a)$ and
Proposition~\ref{field-separable-case}.
\end{proof}

\begin{corollary}\label{proj-simple}
Let $A\subset B=A[b]$ be a projective simple integral ring
extension. Then $A\subset B$ is free and $1,b,\ldots ,b^{n-1}$ is a
basis, where
\begin{eqnarray*}
n=\id_A(b)=\ud_A(B)=\mu_A(B)\espai\mbox{ and
}\espai n=\id_K(b/1)=\ud_K(L)=[L:K].
\end{eqnarray*}
\end{corollary}
\begin{proof}
If $S=A\setminus\{0\}$, then $K=S^{-1}A$ and
$L=S^{-1}B=S^{-1}A[b]=K[b/1]$. Thus $K\subset L=K[b/1]$ is a simple
algebraic field extension.  By Proposition~\ref{first-prop},
$\id_A(b)=\ud_A(B)=\mu_A(B)=n$, say, and
$\id_K(b/1)=\ud_K(L)=[L:K]=m$, say. By Theorem~\ref{proj-klsimple} and
Corollary~\ref{free-klsimple}, $n=m$ and $A\subset B$ is free (see
Corollary~\ref{free-klsimple}).  Since $\{1,b,\ldots,b^{n-1}\}$ is a
minimal system of generators of $B=A[b]$, then it is a basis of the
free $A$-module $B$ (see Reminder~\ref{def-rank}).
\end{proof}

Sub-multiplicativity holds in the case of projective finite ring
extensions $A\subset B$ with $K\subset L$ being simple.

\begin{corollary}\label{sm-projective}
Let $A\subset B$ and $B\subset C$ be two finite ring extensions. If
$A\subset B$ is projective and $K\subset L$ is simple, then
\begin{eqnarray*}
\ud_A(C)\leq \ud_A(B)\ud_B(C).
\end{eqnarray*}
If moreover, $K\subset L$ is separable, $B\subset C$ is projective and
$L\subset M$ is simple, then
\begin{eqnarray*}
\ud_A(C)=\ud_A(B)\ud_B(C).
\end{eqnarray*}
\end{corollary}
\begin{proof}
By Proposition~\ref{local-invariant}, there exists a maximal ideal
$\mfm$ of $A$ such that
$\ud_A(C)=\ud_{A_\mathfrak{m}}(C_{\mathfrak{m}})$.  Since $A\subset B$
is projective, then $A_{\mathfrak{m}}\subset B_{\mathfrak{m}}$ is free
with fields of fractions $Q(A_{\mathfrak{m}})=Q(A)=K$ and
$Q(B_{\mathfrak{m}})=Q(B)=L$, respectively, where $K\subset L$ is
simple by hypothesis. By Corollary~\ref{free-klsimple},
$\ud_{A_\mathfrak{m}}(B_{\mathfrak{m}})=\mu_{A_\mathfrak{m}}(B_{\mathfrak{m}})$. Therefore,
by Corollary~\ref{sm-mid} and Proposition~\ref{first-prop},
\begin{eqnarray*}
\ud_A(C)=\ud_{A_\mathfrak{m}}(C_{\mathfrak{m}})\leq
\ud_{A_\mathfrak{m}}(B_{\mathfrak{m}})
\ud_{B_\mathfrak{m}}(C_{\mathfrak{m}})\leq \ud_A(B)\ud_B(C).
\end{eqnarray*}
As for the second part of the statement, by hypothesis, $A\subset C$
is projective and $K\subset M$ is simple (again, we use the extended
primitive element theorem). By Theorem~\ref{proj-klsimple},
$\ud_K(L)=\ud_A(B)$, $\ud_L(M)=\ud_B(C)$ and $\ud_K(M)=\ud_A(C)$. By
Remark~\ref{mult-fields}, $\ud_K(M)=\ud_K(L)\ud_L(M)$, so
$\ud_A(C)=\ud_A(B)\ud_B(C)$.
\end{proof}

Now we can complement Example~\ref{ring-of-invariants}. Let $A\subset
B$ a ring extension. Let $G$ be a finite group acting as $A$-algebra
automorphisms on $B$. Define $B^G$ as the subring $B^G=\{b\in B\mid
\sigma(b)=b,\mbox{ for all }\sigma\in G\}$. It is said that $A\subset
B$ is a {\em Galois extension with group $G$} if $B^G=A$, and for any
maximal ideal $\mfn$ in $B$ and any $\sigma\in G\setminus \{1\}$,
there is a $b\in B$ such that $\sigma(b)-b\not\in\mfn$ (see, e.g.,
\cite[Definition~4.2.1]{jly}).

\begin{corollary}\label{galois}
Let $G$ be a finite group and let $A\subset B$ be a Galois ring
extension with group $G$. Then $A\subset B$ is a projective finite
ring extension and $\ud_K(L)=\ud_A(B)=[L:K]=\mro(G)$.
\end{corollary}
\begin{proof}
Since $A\subset B$ is a Galois ring extension of domains with group
$G$, then $A\subset B$ is a projective finite ring extension,
$K\subset L$ is a Galois field extension with group $G$ and
$[L:K]=\mro(G)$ (see, e.g., \cite[Subsequent Remark to
  Definition~4.2.1 and Lemma~4.2.5]{jly}). In particular, $K\subset L$
is separable and hence simple. By Theorem~\ref{proj-klsimple},
$\ud_K(L)=\ud_A(B)=[L:K]=\mro(G)$.
\end{proof}

Next we calculate the integral degree when $A$ is a Dedekind domain
and $K\subset L$ is simple, for example, when $B$ is the ring of
integers of an algebraic number field (see Remark~\ref{max-int-deg}).

\begin{corollary}\label{dedekind}
Let $A\subset B$ be a finite ring extension. Suppose that $A$ is
Dedekind and that $K\subset L$ is simple. Then $A\subset B$ is
projective and $\ud_K(L)=\ud_A(B)=\rank_A(B)=[L:K]$. If moreover $A$
is a principal ideal domain, then $A\subset B$ is free and
$\ud_K(L)=\ud_A(B)=\mu_A(B)$.
\end{corollary}
\begin{proof}
From the structure theorem of finitely generated modules over a
Dedekind domain, and since $B$ is a torsion-free $A$-module, it
follows that $A\subset B$ is a projective finite ring extension (see,
e.g., \cite[Corollary to Theorem~1.32, page~30]{narkiewicz}). Since
$A\subset B$ is projective finite and $K\subset L$ is simple, then
$\ud_K(L)=\ud_A(B)=\rank_A(B)=[L:K]$ (see
Theorem~\ref{proj-klsimple}). Finally, if $A$ is a principal ideal
domain, then $A\subset B$ must be free and we apply
Corollary~\ref{free-klsimple}.
\end{proof}

\section{Integrally closed base ring}\label{icd}

As always, $A\subset B$ and $B\subset C$ are integral ring extensions
of domains, and $K$, $L$ and $M$ are their fields of fractions,
respectively. Recall that $\overline{A}$ denotes the integral closure
of $A$ in $K$. In this section we focus our attention on the case
where $A$ is integrally closed. We begin by noting that, in such a
situation, $A\subset B$ has minimal integral degree.

\begin{proposition}\label{kronecker}
Let $A\subset B$ be an integral ring extension. Then, for every $b\in
B$, $\id_K(b)=\id_{\overline{A}}(b)$. In particular, if $A$ is
integrally closed, then $\ud_K(L)=\ud_A(B)$.
\end{proposition}
\begin{proof}
Since $K\supset \overline{A}$, $\id_K(b)\leq
\id_{\overline{A}}(b)$. On the other hand, it is well-known that the
minimal polynomial of $b$ over $K$ has coefficients in $\overline{A}$
(see, e.g., \cite[Chapter~V, \S~1.3, Corollary to
  Proposition~11]{bourbaki-ac}), which forces
$\id_{\overline{A}}(b)\leq \id_K(b)$. So $\id_K(b)=\id_{\overline{A}}(b)$.

Suppose now that $A$ is integrally closed. Then, for every $b\in B$,
$\id_A(b)=\id_{\overline{A}}(b)=\id_K(b)\leq \ud_K(L)$. Thus
$\ud_A(B)\leq \ud_K(L)$. The equality follows from
Proposition~\ref{first-prop}.
\end{proof}

Certainly, $\id_{\overline{A}}(b)$ may not be equal to $\id_A(b)$, as
the next example shows.

\begin{example}
Let $A=\mbz[\sqrt{-3}]$. Then $K=Q(A)=\mbq(\sqrt{-3})$. Let
$b=(1+\sqrt{-3})/2\in K$. Clearly, $b$ is integral over $A$, and the
minimal polynomial of $b$ over $A$ is $T^2-T+1$. Thus $\id_A(b)=2$,
whereas $\id_K(b)=1$.
\end{example}

Recall that a simple integral ring extension $B=A[b]$ over an
integrally closed domain $A$ is free. Indeed, as said above, the
minimal polynomial $p(T)$ of $b$ over $K$ has coefficients in
$A$. Therefore $1,b,\ldots ,b^{n-1}$ is a set of generators of the
$A$-module $A[b]$ (where $n=\deg p(T)$). Moreover, since they are
linearly independent over $K$, they are also linearly independent over
$A$. The next result, which is a rephrasing of this
fact, is obtained as a direct consequence of
Proposition~\ref{kronecker}.

\begin{corollary}\label{rephrasing}
Let $A\subset B$ be a finite ring extension. Suppose that $A$ is an
integrally closed domain. Then, $\ud_A(B)=\mu_A(B)$ is equivalent to
$A\subset B$ free and $K\subset L$ simple. In particular, if $A\subset
B$ is simple and $A$ is integrally closed, then $A\subset B$ is free.
\end{corollary}
\begin{proof}
By Proposition~\ref{kronecker}, one has $\ud_K(L)=\ud_A(B)$. Thus,
$\ud_A(B)=\mu_A(B)$ is equivalent to $\ud_K(L)=[L:K]=\mu_A(B)$ (see
Notation~\ref{diagram}). The latter is equivalent to $A\subset B$ free
and $K\subset L$ simple (see Corollary~\ref{free-klsimple}). To finish
apply Proposition~\ref{first-prop}.
\end{proof}

This corollary suggests how to find a finite integral extension
$A\subset B$ with $\ud_K(L)=\ud_A(B)$ and $[L:K]<\mu_A(B)$. It
suffices to take, as in the next example, an extension of number
fields $K\subset L$ which does not admit a relative integral basis
(see also Final Comments~\ref{finalcomments}).

\begin{example}\label{notuptodown}
Let $K=\mbq(\sqrt{-14})$ and $L=K(\sqrt{-7})$. Let $A$ be the integral
closure of $\mbz$ in $K$ and let $B$ be the integral closure of $\mbz$
in $L$. Then $A\subset B$ is a finite integral extension, $A$ is
integrally closed, $K\subset L$ is simple, but $A\subset B$ is not
free (see \cite{ms}).  Hence, by Corollary~\ref{rephrasing},
$\ud_A(B)<\mu_A(B)$.  Note that $\ud_K(L)=\ud_A(B)=[L:K]=2$. Moreover,
it is well-known that $A=\mbz[\sqrt{-14}]$ and
$B=\mbz[(1+\sqrt{-7})/2,\sqrt{2}]$ (see, e.g.,
\cite[Theorem~9.5]{janusz}). An easy calculation shows that $B=\langle
1,(1+\sqrt{-7})/2,\sqrt{2}\rangle_A$. Thus $\mu_A(B)=3$.
\end{example}

Now, we return to the sub-multiplicativity question.

\begin{theorem}\label{domains}
Let $A\subset B$ and $B\subset C$ be two integral ring extensions.
Then, for every $\alpha\in C$,
\begin{eqnarray*}
\id_A(\alpha)\leq\mu_A(\overline{A})\ud_A(B)\id_B(\alpha).
\end{eqnarray*}
In particular,
\begin{eqnarray*}
\ud_A(C)\leq \mu_A(\overline{A})\ud_A(B)\ud_B(C).
\end{eqnarray*}
\end{theorem}
\begin{proof}
Let $\alpha\in C$. Consider the integral extensions $A\subset
\overline{A}$ and $\overline{A}\subset \overline{A}[C]$, where
$\overline{A}[C]$ stands for the $\overline{A}$-algebra generated by
the elements of $C$. By Theorem~\ref{m-equal-mu}, $\id_A(\alpha)\leq
\mu_A(\overline{A})\id_{\overline{A}}(\alpha)$. But, by
Proposition~\ref{kronecker}, $\id_{\overline{A}}(\alpha)=\id_K(\alpha)$. On
the other hand, applying Theorem~\ref{sm-afe} and
Proposition~\ref{first-prop}, we have
\begin{eqnarray*}
\id_K(\alpha)\leq \ud_K(L)\id_L(\alpha)\leq \ud_A(B)\id_B(\alpha).
\end{eqnarray*}
Hence, $\id_A(\alpha)\leq \mu_A(\overline{A})\ud_A(B)\id_B(\alpha)$.
By Lemma~\ref{equivalences}, we are done.
\end{proof}

\begin{remark}
The ring $\overline{A}[C]$ appears in the proof of
Theorem~\ref{domains}. A natural question is whether this ring is the
tensor product $\overline{A}\otimes_AC$. Observe that indeed there is
a natural surjective morphism of rings $\overline{A}\otimes_AC\to
\overline{A}[C]$. However this morphism is not necessarily an
isomorphism. For instance, take $A=k[t^2,t^3]$ and $C=\overline{A}$,
where $\overline{A}=k[t]$. So $\overline{A}[C]=\overline{A}=k[t]$. One
can check that $\overline{A}\otimes_AC$ is not a domain. Indeed, write
$\overline{A}=A[X]/I$, with $I=(X^2-t^2,
t^2X-t^3,t^3X-t^4)$. Therefore $\overline{A}\otimes_AC=A[X,Y]/H$,
where $H=(X^2-t^2, t^2X-t^3,t^3X-t^4,Y^2-t^2,
t^2Y-t^3,t^3Y-t^4)$. Note that $X^2-Y^2$ is in $H$, but neither $X-Y$
nor $X+Y$ are in $H$. Hence $\overline{A}\otimes_AC$ is not a domain
and cannot be isomorphic to $\overline{A}[C]=k[T]$, which is a domain.
\end{remark}

As an immediate consequence of Theorem~\ref{domains}, we get the
sub-multiplicativity of the integral degree with respect to $A\subset
B$ when $A$ is integrally closed.

\begin{corollary}\label{sm-icd}
Let $A\subset B$ and $B\subset C$ be two integral ring extensions.
Suppose that $A$ is an integrally closed domain. Then
\begin{eqnarray*}
\ud_A(C)\leq \ud_A(B)\ud_B(C).
\end{eqnarray*}
\end{corollary}

However, in the non-integrally closed case, this formula may fail
already for noetherian domains of dimension $1$, as shown below. To
see this, we take advantage of an example due to Dedekind of a
non-monogenic number field $L$. Concretely, we consider $B$ as the
ring of integers of $L$ and find $A$ and $C$ such that
$\ud_A(C)>\ud_A(B)\ud_B(C)$.

\begin{example}\label{no-sm}
Let $\gamma_1$ be a root of the irreducible polynomial
$T^3-T^2-2T-8\in \mbq [T]$. Let $L=\mbq(\gamma_1)$. Let $B$ be the
integral closure of $\mbz$ in $L$ (i.e., the ring of integers of
$L$). Then
\begin{itemize}
\item[$(a)$] $B$ is a free $\mbz$-module with basis
  $\{1,\gamma_1,\gamma_2\}$, where $\gamma_2=(\gamma_1^2+\gamma_1)/2$.
\item[$(b)$]
  $\ud_{\mathbb{Q}}(L)=\ud_{\mathbb{Z}}(B)=\mu_{\mathbb{Z}}(B)=3$ and
  the extension $\mbz\subset B$ is not simple ($L$ is non-monogenic).
\end{itemize}
Let $A=\langle 1,2\gamma_1,2\gamma_2 \rangle_{\mbz}=\{
a+b\gamma_1+c\gamma_2 \in B\mid a,b,c\in\mbz, b\equiv c\equiv 0
\pmod{2} \}$. Then
\begin{itemize}
\item[$(c)$] $A$ is a free $\mbz$-module and an integral domain with
  field of fractions $K=Q(A)=L$.
\item[$(d)$] $B$ is the integral closure of $A$ in $L$, $\ud_A(B)=2$
  and $\mu_A(B)=3$.
\end{itemize}
Let $C=B[\alpha]$, where $\alpha$ is a root of
$p(T)=T^2+\gamma_1T+(1+\gamma_2)\in B[T]$. Then
\begin{itemize}
\item[$(e)$] $B\subset C$ is an integral extension with $\ud_B(C)=2$
  and $\ud_A(C)=6$. 
\end{itemize}
In particular,
$\ud_A(B)\ud_B(C)<\ud_A(C)<\ud_A(\overline{A})\ud_A(B)\ud_B(C)$.
\end{example}
\begin{proof}
By Corollary~\ref{dedekind}, $\mbz\subset B$ is free and
$\ud_{\mathbb{Q}}(L)=\ud_{\mathbb{Z}}(B)=\mu_{\mathbb{Z}}(B)$. Moreover,
since $\gamma_1\in B$ with $\id_{\mathbb{Z}}(\gamma_1)=3$, then
$\ud_{\mathbb{Z}}(B)\geq 3$. The proof that $\{1,\gamma_1,\gamma_2\}$
is a free $\mbz$-basis of $B$ and that $\mbz\subset B$ is not simple
is due to Dedekind (see, e.g., \cite[p.~64]{narkiewicz}).  This proves
$(a)$ and $(b)$.

Note that, from the equalities
\begin{eqnarray*}
\gamma_1^2=-\gamma_1+2\gamma_2\mbox{ ,
}\gamma_2^2=6+2\gamma_1+3\gamma_2\mbox{ and }\gamma_1
\gamma_2=4+2\gamma_2,
\end{eqnarray*}
 the product in $B$ can be immediately computed in terms of its
 $\mbz$-basis $\{1,\gamma_1,\gamma_2\}$.

Clearly $\{1,2\gamma_1,2\gamma_2\}$ are $\mbz$-linearly independent.
One can easily check that $A$ is a ring and that $x^2+x\in A$, for
every $x\in B$. Hence, $A\subset B$ is an integral extension with
$\ud_A(B)=2$. Moreover, the field of fractions of $A$ is $K=Q(A)=L$,
and the integral closure of $A$ in $K$ is $B$. Observe that
$\mu_A(B)\leq \mu_{\mathbb{Z}}(B)=3$. Below we will see that
$\mu_A(B)=3$.

Now let us prove that $\ud_B(C)=2$. One readily checks that the
discriminant $\Delta=-\gamma_1^2-2\gamma_1-4$ of $p(T)$ has norm
$N_{L/\mathbb{Q}}(\Delta)=-16$. Since $-16$ is not a square in $\mbq$,
then $\Delta$ cannot be a square in $L$. Therefore $p(T)$ is
irreducible over $L$ and $\ud_B(C)=2$.

Let $h(T)\in A[T]$ be a minimal degree polynomial of $\alpha$ over
$A$. Since $p(T)$ is the irreducible polynomial of $\alpha$ over $L$,
it follows that $h(T)=p(T)q(T)$, for some $q(T)\in L[T]$. Moreover,
$q(T)$ must necessarily belong to $B[T]$, because $B$ is integrally
closed in $L$ (see, e.g., \cite[Chapter~V, \S~1.3,
  Proposition~11]{bourbaki-ac}). Therefore, $q(T)$ is a monic
polynomial in $B[T]$ such that $p(T)q(T)\in A[T]$. An easy computation
shows that this implies that $\deg(q(T))\geq 4$ (note that the
existence of such a polynomial $q(T)=T^n+b_{1}T^{n-1}+\cdots
+b_{n-1}T+b_n$ is equivalent to the solvability in $\mbz$ modulo $2$
of a certain system of linear equations with coefficients in $\mbz$,
in the unknowns $a_{ij}\in\mbz$, where
$b_i=a_{i,1}+a_{i,2}\gamma_1+a_{i,3}\gamma_2$).  Thus,
$\id_A(\alpha)=\deg(h(T))\geq 6$. By Theorem~\ref{m-equal-mu},
\begin{eqnarray*}
6\leq\id_A(\alpha)\leq\ud_A(C)\leq \mu_A(B)\ud_B(C)\leq 6.
\end{eqnarray*}
Hence $\ud_A(C)=6$ and $\mu_A(B)=3$.
\end{proof}

\begin{remark}\label{notpossiblerank2}
It is not possible to construct a similar example with $B$ having rank
$2$ over $\mbz$, because $\ud_{A}(B)\leq \mu_A(B)\leq
\rank_{\mathbb{Z}}(B)=2$ implies $\ud_A(B)=\mu_A(B)$ and then, by
Corollary~\ref{sm-mid}, $\ud_A(B)\leq \ud_A(B)\ud_B(C)$.
\end{remark}

\section{Upper-semicontinuity}\label{semi}

Recall that $A\subset B$ is an integral ring extension of integral
domains, and $K=Q(A)$ and $L=Q(B)$ are their fields of fractions. Let
$\ud:\Spec(A)\to\mbn$ be defined by
$\ud(\mfp)=\ud_{A_{\mathfrak{p}}}(B_{\mathfrak{p}})$. In this section
we study the upper-semicontinuity of $\ud$, that is, whether or not,
\begin{eqnarray*}
\ud^{-1}([n,+\infty)) =\{\mfp\in\Spec(A)\mid \ud(\mfp)\geq n\}
\end{eqnarray*}
is a closed set for every $n\geq 1$. There are two cases in which
upper-semicontinuity follows easily from our previous results.

\begin{proposition}\label{semi-two}
Let $A\subset B$ be an integral ring extension. Then
$\ud:\Spec(A)\to\mbn$, defined by
$\ud(\mfp)=\ud_{A_{\mathfrak{p}}}(B_{\mathfrak{p}})$, is
upper-semicontinuous in any of the following cases:
\begin{itemize}
\item[$(a)$] $A\subset B$ is simple; 
\item[$(b)$] $A\subset B$ has minimal integral degree (e.g., $A\subset
  B$ is projective finite and $K\subset L$ is simple; or $A$ is
  integrally closed).
\end{itemize}
\end{proposition}
\begin{proof}
If $A\subset B=A[b]$ is simple, then $A_{\mathfrak{p}}\subset
B_{\mathfrak{p}}=A_{\mathfrak{p}}[b/1]$ is simple too, for every
$\mfp\in\Spec(A)$.  By Proposition~\ref{first-prop}, it follows that
$\ud(\mfp)=\ud_{A_\mathfrak{p}}(B_{\mathfrak{p}})=
\mu_{A_\mathfrak{p}}(B_{\mathfrak{p}})$. But the minimal number of
generators is known to be an upper-semicontinuous function (see, e.g.,
\cite[Chapter~IV, \S~2, Corollary~2.6]{kunz}). This shows case $(a)$.
By Proposition~\ref{first-prop},~$(e)$, $\ud_K(L)\leq
\ud_{A_\mathfrak{p}}(B_{\mathfrak{p}})\leq\ud_A(B)$, for every
$\mfp\in\Spec(A)$. In case $(b)$, that is, if $\ud_K(L)=\ud_A(B)$,
then $\ud(\mfp)= \ud_{A_\mathfrak{p}}(B_{\mathfrak{p}})=\ud_{K}(L)$ is
constant, and thus upper-semicontinuous. This happens, for instance,
if $A\subset B$ is projective finite and $K\subset L$ is simple or $A$
is integrally closed, then $\ud_K(L)=\ud_A(B)$ (see
Theorem~\ref{proj-klsimple} or Proposition~\ref{kronecker}).
\end{proof}

A possible way to weaken the integrally closed hypothesis is to shrink
the conductor $\mcc=(A:\overline{A})$ of $A$ in its integral closure
$\overline{A}$. A first thought would be to suppose that $\mcc$ is of
maximal height. However, with some extra assumptions on $A$, e.g., $A$
local Cohen-Macaulay, analytically unramified and $A$ not integrally
closed, one can prove that the conductor must have height $1$ (see,
e.g., \cite[Exercise~12.6]{hs}). In this sense, it seems appropriate
to start by considering when $\dim A=1$.

\begin{theorem}\label{semi-nagata}
Let $A\subset B$ be an integral ring extension. Suppose that $A$ is a
noetherian domain of dimension $1$ and with finite integral closure
(e.g., $A$ is a Nagata ring). Then $\ud:\Spec(A)\to\mbn$, defined by
$\ud(\mfp)=\ud_{A_{\mathfrak{p}}}(B_{\mathfrak{p}})$, is
upper-semicontinuous.
\end{theorem}
\begin{proof}
If $A$ is integrally closed, the result follows from
Proposition~\ref{semi-two}. Thus we can suppose that $A$ is not
integrally closed. Since $\overline{A}$ is finitely generated as an
$A$-module, then $\mcc=(A:\overline{A})\neq 0$. Since $A$ is a one
dimensional domain, $\mcc$ has height $1$ and any prime ideal $\mfp$
containing $\mcc$ must be minimal over it. Therefore, the closed set
$V(\mcc)$ coincides with the set of minimal primes over $\mcc$, so it
is finite. Note that, for any $\mfp\in\Spec(A)$, $\ud(\mfp)=
\ud_{A_\mathfrak{p}}(B_{\mathfrak{p}})\geq \ud_{K}(L)$ (see
Proposition~\ref{first-prop}). Moreover, if $\mfp\not\in V(\mcc)$,
then $A_{\mathfrak{p}}=\overline{A_{\mathfrak{p}}}$ and
$\ud(\mfp)=\ud_K(L)$ (see Proposition~\ref{kronecker}).  Now, take
$n\geq 1$. If $n>\ud_K(L)$, then $\{\mfp\in\Spec(A)\mid \ud(\mfp)\geq
n\}\subseteq V(\mcc)$ is a finite set, hence a closed set. If $n\leq
\ud_K(L)$, then $\{\mfp\in\Spec(A)\mid \ud(\mfp)\geq n\}=\Spec(A)$.
Thus, for every $n\geq 1$, $\ud^{-1}([n,+\infty))$ is a closed set and
  $\ud:\Spec(A)\to\mbn$ is upper-semicontinuous.
\end{proof}

\begin{remark}
Note that the proof of Theorem~\ref{semi-nagata} only uses that
$V(\mcc)$ is a finite set of $\Spec(A)$. For instance, it also holds
if $A$ is a noetherian local domain of dimension $2$ and with finite
integral closure $\overline{A}$. Another example where it would work
would be the following: let $A$ be the coordinate ring of a reduced
and irreducible variety $V$ over a field of characteristic zero. Then
the conductor $\mcc$ contains the Jacobian ideal $J$. Now $J$ defines
the singular locus of $V$, so if we suppose that $V$ has only isolated
singularities, then $J$ is of dimension zero, so $\mcc$ is of
dimension zero also. Hence $V(\mcc)$ is finite (see
\cite[Theorem~4.4.9]{hs} and \cite[Corollary~6.4.1]{vasconcelos}).
\end{remark}

If we skip the condition that $\overline{A}$ be finitely generated,
the result may fail. The following example is inspired by
\cite[Example~1.4]{sv} (see also \cite[Example~6.6.]{gp}).

\begin{example}\label{no-semi}
There exists a noetherian domain $A$ of dimension $1$ with
$\ud_A(\overline{A})=2$, but $\mu_A(\overline{A})=\infty$, and such
that $\ud:\Spec(A)\to\mbn$, defined by
$\ud(\mfp)=\ud_{A_{\mathfrak{p}}}(\overline{A}_{\mathfrak{p}})$, is
not upper-semicontinuous.
\end{example}
\begin{proof}
Let $t_1,t_2,\ldots ,t_n,\ldots $ be infinitely many indeterminates
over a field $k$. Let
\begin{eqnarray*}
R=k[t_1^2,t_1^3,t_2^2,t_2^3,\ldots ]\subset D=k[t_1,t_2,\ldots ].
\end{eqnarray*}
Clearly $\overline{R}=D$. Note that for $f\in D=k[t_1,t_2,\ldots ]$,
$f\in R$ if and only if every monomial $\lambda t_1^{i_1}\cdots
t_r^{i_r}$ of $f$ has each $i_j=0$ or $i_j\geq 2$.

For every $n\geq 1$, let $\mfq_n=(t_n^2,t_n^3)R$, which is a prime
ideal of $R$ of height $1$. Note that for $f\in R$, $f\in \mfq_n$ if
and only if every monomial $\lambda t_1^{i_1}\cdots t_r^{i_r}$ of $f$
has $i_n\geq 2$. It follows that $t_n\not\in R_{\mathfrak{q_n}}$,
because if $t_n=a/b$, $a,b\in R$ and $b\not\in\mfq_n$, then every
monomial of $a=bt_n$ has each $i_j=0$ or $i_j\geq 2$, so has $i_n\geq
2$. Therefore, $t_n$ appears in each monomial of $b$, but since $b\in
R$, the exponent of $t_n$ in each monomial of $b$ must be at least
$2$, so $b\in\mfq_n$, a contradiction.

Now, set $R\subset D_n=k[t_1,\ldots
  ,t_{n-1},t_n^2,t_n^3,t_{n+1},\ldots ]\subset D$ and
$S_n=R\setminus\mfq_n$, a multiplicatively closed subset of
$R$. Clearly $R_{\mathfrak{q_n}}=S_n^{-1}D_n$.

\noindent {\sc Claim}. Let $I$ be an ideal of $R$ such that $I\subseteq
\cup_{n\geq 1}\mfq_n$. Then $I$ is contained in some $\mfq_j$.

\noindent If $I$ is contained in a finite union of $\mfq_i$, using the
ordinary prime avoidance lemma, we are done. Suppose that $I$ is not
contained in any finite union of $\mfq_i$ and let us reach a
contradiction. Take $f\in I$, $f\neq 0$. Then $f\in k[t_1,\ldots
  ,t_n]$ for some $n\geq 1$ and $f$ is in a finite number of $\mfq_i$,
corresponding to the variables $t_i$ that appear in every single
monomial of $f$. We can suppose that $f\in\mfq_1\cap \ldots\cap
\mfq_r$, for some $1\leq r\leq n$, and $f\not\in\mfq_i$, for
$i>r$. Since $I\not\subset \mfq_1\cup\ldots \cup \mfq_r$, there exists
$g\in I$ such that $g\not\in \mfq_1\cup\ldots \cup \mfq_r$.  Let
$h=t_s^2g\in I$, where $s>n$, so that $f$ and $h$ have no common
monomials. Since $\mfq_i$ are prime, then $h=t_s^2g\not\in
\mfq_1\cup\ldots \cup\mfq_r$. Since $f,h\in I\subseteq \cup_{n\geq
  1}\mfq_n$, then $f+h\in \mfq_m$, for some $m\geq 1$. But since $f\in
\mfq_1\cap \ldots\cap \mfq_r$ and $h\not\in \mfq_1\cup\ldots
\cup\mfq_r$, then necessarily $m>r$. Thus $f+h\in\mfq_m$, where
$m>r$. But since $f$ and $h$ have no common monomials, this implies
that every monomial of $f$ must contain $t_m^2$, so $f\in\mfq_m$, a
contradiction. Hence $I\subseteq \mfq_j$ for some $j$ and the {\sc
  Claim} is proved. (An alternative proof would follow from
\cite[Proposition~2.5]{ShV}, provided that $k$ is uncountable.)

Let $S=R\setminus \cup_{n\geq 1}\mfq_n$, a multiplicatively closed
subset of $R$. Let $A=S^{-1}R$ and $\mfp_n=S^{-1}\mfq_n$. If $Q$ is a
prime ideal of $R$ such that $Q\subseteq \cup_{n\geq 1}\mfq_n$, then,
by the {\sc Claim} above, $Q\subseteq \mfq_j$, for some $n\geq 1$. In
particular, $\Spec(A)=\{(0)\}\cup \{\mfp_n\mid n\geq 1\}$, where each
$\mfp_n$ is finitely generated. Therefore $A$ is a one dimensional
noetherian domain.

For every $n\geq 1$,
$A_{\mathfrak{p}_n}=(S^{-1}R)_{S^{-1}\mathfrak{q}_n}=
R_{\mathfrak{q}_n}=S^{-1}_{n}D_n$. Moreover, $t_n=t_n^3/t_n^2$ is in
the field of fractions of $A_{\mathfrak{p}_n}$ and $t_n^2\in
A_{\mathfrak{p}_n}$, i.e. $t_n$ is integral over
$A_{\mathfrak{p}_n}$. Thus
\begin{eqnarray*}
A_{\mathfrak{p}_n}[t_n]= (S_n^{-1}D_n)[t_n]=S_n^{-1}D\espai\mbox{ and
}\espai
\overline{A_{\mathfrak{p}_n}}=\overline{A_{\mathfrak{p}_n}[t_n]}=
\overline{S_n^{-1}D}= S_n^{-1}(\overline{D})=S_n^{-1}D.
\end{eqnarray*}
Hence $\overline{A_{\mathfrak{p}_n}}=A_{\mathfrak{p}_n}[t_n]$. Recall
that $t_n\not\in R_{\mathfrak{q}_n}=A_{\mathfrak{p}_n}$ and
$\ud_{A_{\mathfrak{p}_n}}(t_n)\leq 2$. By Proposition~\ref{first-prop},
$\ud_{A_{\mathfrak{p}_n}}( \overline{A_{\mathfrak{p}_n}})=
\ud_{A_{\mathfrak{p}_n}}(A_{\mathfrak{p}_n}[t_n])=2$.

Consider the integral extension $A\subset \overline{A}$ and
$\ud:\Spec(A)\to\mbn$, defined by
$\ud(\mfp)=\ud_{A_{\mathfrak{p}}}(\overline{A}_{\mathfrak{p}})=
\ud_{A_{\mathfrak{p}}}(\overline{A_{\mathfrak{p}}})$. We have just
shown that, for every $n\geq 1$,
$\ud(\mfp_n)=\ud_{A_{\mathfrak{p}_n}}(
\overline{A_{\mathfrak{p}_n}})=2$. On the other hand,
$\ud((0))=\ud_{Q(A)}(Q(\overline{A}))=1$ because
$Q(A)=Q(\overline{A})$.  Therefore,
$\ud^{-1}([2,+\infty))=\Spec(A)\setminus\{(0)\}$, which is not a
  closed set. Indeed, suppose that $\Spec(A)\setminus\{(0)\}=V(I)$,
  for some non-zero ideal $I$. Since $A$ is a one dimensional
  noetherian domain, $I$ has height $1$ and $V(I)$ is the finite set
  of associated primes to $I$. However, $\Spec(A)\setminus
  \{(0)\}=\Max(A)$, which is infinite, a contradiction. So
  $\ud:\Spec(A)\to\mbn$ is not upper-semicontinuous.
\end{proof}

\begin{remark}\label{dim1-notsm}
Contrary to the upper-semicontinuity, sub-multiplicativity does not
work for one dimensional noetherian domains with finite integral
closure. See Example~\ref{no-sm}, where $A$ was a noetherian domain of
dimension $1$ and with finite integral closure.
\end{remark}

\begin{finalcomments}\label{finalcomments}
We finish the paper by mentioning some points that we think would be
worth clarifying. To simplify, suppose that $A\subset B$ and $B\subset
C$ are two finite ring extensions, where, as always, $A$ and $B$ are
two integral domains, and $K$ and $L$ are their fields of fractions,
respectively.
\begin{itemize}
\item[$(1)$] We have shown that $A\subset B$ of maximal integral
  degree implies sub-multiplicativity
  (cf. Corollary~\ref{sm-mid}). Does the same work for minimal
  integral degreee?
\item[$(2)$] We have shown that $A\subset B$ of minimal integral
  degree implies upper-semicontinuity
  (cf. Proposition~\ref{semi-two}). Does the same work for maximal
  integral degreee?
\item[$(3)$] We have shown that $A\subset B$ free and $K\subset L$
  simple implies $\ud_K(L)=\ud_A(B)$ (Corollary~\ref{free-klsimple}).
  Can we omit the hypothesis $K\subset L$ simple? In other words, does
  $[L:K]=\mu_A(B)$ imply $\ud_K(L)=\ud_A(B)$?  If so, we would have a
  ``down-to-up rigidity'' in the diagram of
  Notation~\ref{diagram}. Note that the ``up-to-down rigidity'' is not
  true (see, e.g., Example~\ref{notuptodown}).
\item[$(4)$] Does the condition $\ud_A(B)=\mu_A(B)$ localize? In
  particular, does $\ud_A(B)=\mu_A(B)$ imply $\ud_K(L)=[L:K]$? That
  would imply a ``right-to-left rigidity'' in the diagram of
  Notation~\ref{diagram}. If $A$ is integrally closed, the answer is
  affirmative. Note that Examples~\ref{d<mu} and \ref{no-sm} affirm
  that the ``left-to-right rigidity'' is not true.
\item[$(5)$] It would be interesting to study the
  sub-multiplicativity and upper-semicontinuity properties for the
  specific case of affine domains $A$ and $B$.
\item[$(6)$] Can one replace $\mu_A(\overline{A})$ by
  $\ud_A(\overline{A})$ in the inequality $\ud_A(C)\leq
  \mu_A(\overline{A})\ud_A(B)\ud_B(C)$ of Theorem~\ref{domains}?
\item[$(7)$] Is the integral degree upper-semicontinuous for Nagata
  rings of dimension greater than $1$?
\item[$(8)$] Is there any clear relationship between $\ud_A(B)$ and
  the pair of numbers $\ud_{A/\mathfrak{p}}(B/\mfp B)$ and
  $\ud_{A_{\mathfrak{p}}}(B_{\mathfrak{p}})$? An affirmative answer
  could be useful in recursive arguments.
\item[$(9)$] Upper-semicontinuity does not imply
  sub-multiplicativity. We wonder to what extent sub-multiplicativity
  could imply upper-semicontinuity.
\end{itemize}
\end{finalcomments}

\section*{Acknowledgement} 
The authors were unaware that references \cite{jacobson-book},
\cite{kaplansky}, \cite{kurosch}, \cite{levitzki} and \cite{voight}
treated some aspects of the notion of integral degree, though in a
different framework. They want to thank the referee for pointing them
out. They also would like to thank the referee for his/her comments.
The third author was partially supported by grant 2014 SGR-634 and
grant MTM2015-69135-P. The fourth author was partially supported by
grant 2014 SGR-634 and grant MTM2015-66180-R.

\vspace*{1cm}

{\footnotesize\sc

\noindent Departament d'\`Algebra i Geometria, Universitat de
Barcelona. \newline Gran Via de les Corts Catalanes 585, E-08007 Barcelona.
{\em giral@ub.edu}

\vspace{0.3cm}

\noindent Maxwell Institute for Mathematical Sciences,
School of Mathematics, University of Edinburgh. \newline
EH9 3DF, Edinburgh.

\vspace{0.3cm}

\noindent Departament de Matem\`atiques, Universitat Polit\`ecnica de
Catalunya. \newline Diagonal 647, ETSEIB, E-08028 Barcelona.  {\em
  francesc.planas@upc.edu}

\vspace{0.3cm}

\noindent Departament de Matem\`atiques, Universitat Polit\`ecnica de
Catalunya. \newline Diagonal 647, ETSEIB, E-08028 Barcelona.  {\em
  bernat.plans@upc.edu}

}

{\small

\begin{thebibliography}{cc}
\bibitem{bourbaki}{N. Bourbaki, Algebra. Chapters 4-7. Translated
  from the 1981 French edition by P. M. Cohn and J. Howie. Reprint of
  the 1990 English edition [Springer, Berlin]. Elements of Mathematics
  (Berlin). Springer-Verlag, Berlin, 2003.}
\bibitem{bourbaki-ac}{N. Bourbaki, Commutative Algebra. Chapters
  1-7. Translated from the French. Reprint of the 1989 English
  translation. Elements of Mathematics (Berlin). Springer-Verlag,
  Berlin, 1998.}
\bibitem{gp}{J.M.~Giral, F. Planas-Vilanova, Integral degree of a
  ring, reduction numbers and uniform Artin-Rees numbers, J. Algebra
  {\bf 319} (2008), 3398-3418.}
\bibitem{hs}{C. Huneke and I. Swanson, Integral closure of ideals,
  rings, and modules. London Mathematical Society Lecture Note Series,
  336. Cambridge University Press, Cambridge, 2006.}
\bibitem{jacobson-annals}{N. Jacobson, Structure theory for algebraic
  algebras of bounded degree. Ann. of Math. {\bf 46} (1945), 695-707.}
\bibitem{jacobson-book}{N. Jacobson, Lectures in abstract
  algebra. III. Theory of fields and Galois theory. Second corrected
  printing. Graduate Texts in Mathematics, 32. Springer-Verlag, New
  York-Heidelberg, 1975.}
\bibitem{janusz}{G.J. Janusz, Algebraic Number Fields, Second Edition,
  Graduate Studies in Mathematics, 7. American Mathematical Society,
  1996.}
\bibitem{jly}{C.U. Jensen, A. Ledet, N. Yui, Generic polynomials.
  Constructive aspects of the inverse Galois problem. Mathematical
  Sciences Research Institute Publications, 45. Cambridge University
  Press, Cambridge, 2002.}
\bibitem{kaplansky}{I. Kaplansky, On a problem of Kurosch and
  Jacobson. Bull. Amer. Math. Soc. {\bf 52} (1946), 496-500.}
\bibitem{kurosch}{A. Kurosch, Ringtheoretische Probleme, die mit dem
  Burnsideschen Problem \"uber periodische Gruppen in Zusammenhang
  stehen, Bull. Acad. Sci. URSS. Sér. Math. [Izvestia Akad. Nauk SSSR]
  {\bf 5} (1941), 233-240.}
\bibitem{kunz}{E. Kunz, Introduction to commutative algebra and
  algebraic geometry. Translated from the German by Michael
  Ackerman. With a preface by David Mumford. Birkh\"auser Boston,
  Inc., Boston, MA, 1985.}
\bibitem{levitzki}{J. Levitzki, On a problem of
  A. Kurosch. Bull. Amer. Math. Soc. {\bf 52} (1946), 1033-1035.}
\bibitem{ms}{R. MacKenzie, John Scheuneman, A number field without a
  relative integral basis. Amer. Math. Monthly {\bf 78} (1971),
  882-883.}
\bibitem{matsumura-ca}{H. Matsumura, Commutative algebra. Second
  edition. Mathematics Lecture Note Series, 56. Benjamin Cummings
  Publishing Co., Inc., Reading, Mass., 1980.}
\bibitem{matsumura-crt}{H. Matsumura, Commutative ring theory. Cambridge
  Studies in Advanced Mathematics, 8. Cambridge University Press,
  Cambridge, 1986.}
\bibitem{narkiewicz}{W. Narkiewicz, Elementary and analytic theory of
  algebraic numbers.  Third edition. Springer Monographs in
  Mathematics. Springer-Verlag, Berlin, 2004.}
\bibitem{sv}{J. Sally, W.V. Vasconcelos, Stable rings.  J. Pure
Appl. Algebra {\bf 4} (1974), 319-336.}
\bibitem{ShV}{R.Y. Sharp, P. V\'amos, Baire's category theorem and
  prime avoidance in complete local rings. Arch. Math. (Basel) {\bf
    44} (1985), no. 3, 243-248. }
\bibitem{vasconcelos}{W.V. Vasconcelos, Computational methods in
  commutative algebra and algebraic geometry. With chapters by David
  Eisenbud, Daniel R. Grayson, J\"urgen Herzog and Michael
  Stillman. Algorithms and Computation in Mathematics,
  2. Springer-Verlag, Berlin, 1998.}
\bibitem{voight}{J. Voight, Rings of low rank with a standard
  involution. Illinois J. Math. {\bf 55} (2011), 1135-1154.}
\end{thebibliography}}
\end{document}